\documentclass[final]{siamltex}

\usepackage{amsmath}
\usepackage{mathrsfs}
\usepackage{amsfonts}
\usepackage[dvips]{graphicx}
\usepackage{amsopn}
\usepackage{amstext}
\usepackage{amssymb}
\usepackage{amscd}
\usepackage{latexsym}
\usepackage{setspace} 

\newtheorem{example}[theorem]{Example}

\DeclareMathOperator{\Prob}{P}
\DeclareMathOperator{\E}{E}

\DeclareMathOperator{\e}{e}

\newcommand{\NN}{\mathbb N}

\newcommand{\comment}[1]{}

\begin{document}

\title{Letter Change Bias and Local Uniqueness in Optimal Sequence 
Alignments} 

\author{Raphael Hauser\footnotemark[2] \;and 
Heinrich Matzinger \footnotemark[3]}

\renewcommand{\thefootnote}{\fnsymbol{footnote}}

\footnotetext[2]{Mathematical Institute, 24–29 St Giles', 
Oxford OX1 3LB, United Kingdom. (hauser@maths.ox.ac.uk).  This work was supported by the 
Engineering and Physical Sciences Research Council [grant number EP/H02686X/1].}
\footnotetext[3]{School of Mathematics, 
Georgia Institute of Technology, 686 Cherry Street, Atlanta, 
GA 30332-0160, USA, (matzi@math.gatech.edu).}

\renewcommand{\thefootnote}{\arabic{footnote}}

\maketitle

\begin{abstract}
Considering two optimally aligned random sequences, we investigate the effect on the 
alignment score caused by changing a random letter in one of the two sequences. Using 
this idea in conjunction with large deviations theory, we show 
that in alignments with a low proportion of gaps the optimal alignment is locally unique 
in most places with high probability. This has implications in the design of recently pioneered 
alignment methods that use the local uniqueness as a homology indicator. 
\end{abstract}



\section{Introduction}
\label{introduction}
The purpose of this paper is to gain insight into the local multiplicity 
of optimal alignments of two random binary sequences. Before introducing the problem setting, 
let us give a brief motivation and literature review. 

\subsection{Motivation}\label{why?}
A fairly general and useful technique to identify high quality alignments of 
two sequences $x=$'$x_1$\dots$x_m$' and $y=$'$y_1$\dots$y_n$' with characters 
from a finite alphabet ${\mathbb A}$ is to consider alignments with gaps $\sqcup$, 
\begin{equation*}
\begin{array}{ccccc}
\sqcup&x_1&x_2&x_3&\sqcup\\
y_1&y_2&\sqcup&y_3&y_4
\end{array}
\end{equation*}
and to quantify the quality of such an alignment with a score of the form 
\begin{equation}\label{scoring function}
S(x,y)=s(\sqcup,y_1)+s(x_1,y_2)+s(x_2,\sqcup)+s(x_3,y_3)+s(\sqcup,y_4).
\end{equation}
When a scoring function of the form \eqref{scoring function} is used, 
a good choice of individual scores $s(a,b)$ of matched symbols $a$ and $b$ 
and the choice of the alphabet ${\mathbb A}$ depend on the specific 
application one has in mind. The matching of any symbol $a\in{\mathbb A}$ 
with a gap $\sqcup$ is typically penalized by a negative score term 
$s(a,\sqcup)$, $s(\sqcup,a)<0$. 

The simplest similarity measure of the form \eqref{scoring function} arises as 
the length of a {\em longest common subsequence} (LCS). A common 
subsequence is any sequence that can be obtained by deleting some characters 
of either sequence and keeping the remaining ones in the original order. The 
length of a longest common subsequence of two sequences $x$ and $y$ is the 
same as the score $S(x,y)$, where individual scoring terms are defined as follows, 
\begin{equation*}
s(a,b)=\begin{cases}1\quad&\text{if }a=b\neq\sqcup,\\
\infty&\text{if }a\neq b\text{ and }a,b\neq\sqcup,\\
0&\text{if }a=\sqcup\text{ or }b=\sqcup,\text{ but not both}.
\end{cases}
\end{equation*}

Sequence alignment techniques play an 
important role in biology 
(see e.g.\ \cite{Watermangeneralintro}), speech recognition, 
pattern recognition, automated translation and other areas where hidden 
Markov models are used as an analytic tool. The study of optimal alignments 
of random sequences is also interesting to statistical physicists, 
because it can be seen as a first passage percolation problem on an oriented 
graph with correlated weights. First passage percolation is a mature field 
and has been a major research area in discrete probability for several 
decades. An excellent overview can be found in the chapter dedicated to 
first passage percolation of Volume 110 of the Springer Encyclopaedia of Mathematical 
Sciences.

Among several long-standing open questions in this field are the problems of 
identifying the exact order of the fluctuations and the proportion of points 
where the optimal alignment is unique. Recently, significant progress has been made 
on both questions: In several special cases it was shown that a positive bias effect 
of a random letter-change on the optimal alignment score of two random sequences 
of length $n$ exists and implies an order of fluctuation proportional to $\sqrt{n}$, 
see \cite{BonettoLCS}, \cite{Lember}, \cite{HoudreLemberMatz}, \cite{LemberMatzAnnals}. 
In \cite{AmsaluHauserMatzinger} it was shown how to apply this principle to 
arbitrary scoring functions, and how to detect the positive bias effect via a nontrivial 
Montecarlo technique. In \cite{Hirmo}, a case study was conducted on using the 
local uniqueness of optimal alignments to detect the homology of two DNA sequences. 
The motivation behind this method is the empirical observation that all optimal alignments of 
two biologically related sequences are identical in most places, while optimal alignments 
of unrelated sequences are locally nonunique in most places. 

Our paper concerns a theoretical study of this last observation. Considering two 
independent random sequences with i.i.d.\ characters, we examine their optimal alignments 
containing a fixed proportion of gaps and prove that when the proportion of gaps is 
small, then with high probability optimal alignments differ only in a small number of places 
and are locally unique everywhere else. Our result implies that the approach of \cite{Hirmo} 
can only work for scoring functions in which gaps are not penalized too heavily, as this 
would force the number of gaps appearing in optimal alignments to be small relative to the 
length of the two sequences. 

Optimal alignments of random sequences are often used as null-models in statistical 
tests to decide on whether two or more given sequences are homologous. The mathematical 
underpinnings are best understood in the context of the LCS setting. Let $L_n$ denote 
the length of the LCS of two independent binary i.i.d.\ sequences of 
length $n$. Using a subadditivity argument, Chv\`atal-Sankoff 
\cite{Sankoff1} showed that the limit
\begin{equation*}
\gamma:=\lim_{n\rightarrow\infty}\frac{E[L_n]}{n}
\end{equation*}
exists. Determining the exact value of $\gamma$ -- the so-called 
{\em Chv\`atal-Sankoff constant} -- is a long standing open problem. 
However, upper and lower bounds are known, see \cite{Sankoff1}, 
\cite{Baeza1999}, \cite{Deken}, \cite{Paterson1,Paterson2}, 
\cite{HauserMartinezMatzinger}, \cite{Alexander}.

Another long open problem is the determination of   
the exact order of the fluctuation of the LCS score as a function of the length of the 
sequences. Considering the case of 
binary sequences obtained by flipping perfect coins, it was shown in \cite{steele86} that 
$VAR[L_n]\leq n$. Montecarlo simulations in \cite{Sankoff1} led to the conjecture that $VAR[L_n]=o(n^{2/3})$. This 
order of magnitude is similar to the order for the so-called 
{\em longest increasing subsequence} (LIS) of random permutations, 
see \cite{BaikDeiftJohansson99} and \cite{Aldous99}). The LIS setting is asymptotically equivalent 
to first passage percolation on a oriented Poisson random graph. 
In \cite{Waterman-estimation} it was 
conjectured that in many cases the variance of $L_n$ grows linearly. This 
seems indeed to be the case generically \cite{AmsaluHauserMatzinger}, but there 
may exist different orders of magnitude for these fluctuations, depending 
on the distribution of the sequences $X$ and $Y$, see \cite{BonettoLCS}, \cite{Lember}, 
\cite{watermanphase}.

\subsection{Problem Setting and Key Ideas}\label{problem setting}
Let us consider the set of optimal 
LCS alignments of the two sequences 'I-do-not-like-symmetry' and 'I-detest-symmetry'. 
If we were to give an exhaustive list of optimal alignments, we would find 
that all of them are of the form 
\begin{equation*}
\begin{array}{ccccccccccccccc}
I&-&d&*&\dots&*&-&s&y&m&m&e&t&r&y\\
I&-&d&*&\dots&*&-&s&y&m&m&e&t&r&y,
\end{array}
\end{equation*}
that is, all optimal LCS alignments agree in the region outside of the 
places filled with wild card stars $*$. 
We say that in this region the optimal alignment is {\em locally unique}. 
In our example the optimal LCS alignment is locally unique on a 
large proportion of the two sequences. 
 
This observation is typical not only in the LCS setting, but for general 
scoring functions $S(x,y)$ as introduced in \eqref{scoring function}: When 
two sequences are closely related 
to one another, the optimal alignments are locally unique in many places. 
Conversely, the optimal alignments of two random sequences with i.i.d.\ entries 
often are locally unique only in very few places. \cite{Hirmo} exploited this 
observation in the design of a homology detecting algorithm for DNA sequences. 
However, as this paper will reveal, for this method to work one has to select the 
gap penalty $s(a,\sqcup)$ quite carefully: When 
gaps are strongly penalized, then no more than a constant proportion of gaps 
are observed in optimal alignments, and if the proportion of gaps is small then 
the optimal alignment is locally unique in most places even for i.i.d.\ 
random (and hence totally unrelated) sequences. 

Our paper concerns a theoretical analysis of this phenomenon. To 
make the analysis transparent, we chose a simplified setup in which 
$X=X_1\dots X_m$ and $Y=Y_1\dots Y_n$ are random sequences consisting of 
i.i.d.\ standard Bernoulli variables, where $m=\lfloor(1-\delta)n\rfloor$ 
depends on $n$ via a fixed gap proportion $\delta$. We then investigate 
alignments of $X$ and $Y$ that contain gaps only in $X$, and we use 
a scoring function in which matching symbols contribute to the 
total score with unit weight and non-matching ones with zero weight. 
Our interest is in the random number $U\leq m$ of indices $i$ for which 
$X_i$ is aligned with more than one $Y_j$ under the different optimal 
alignments. The main theorem of this paper shows that when $\delta$ 
is small and $n$ is large, optimal alignments are locally unique in 
an arbitrarily large proportion of places with arbitrarily high 
probability:

\begin{theorem}\label{maintheorem}
For all $\varepsilon>0$ there exists $\delta_0>0$ and $n_0\in\NN$ 
such that for all $\delta\in(0,\delta_0)$ and $n>n_0$, 
$\Prob\left[U\geq m\varepsilon\right]<\varepsilon$.\\
\end{theorem}

While it is clear that $\Prob[U>m\varepsilon]=0$ when $\delta=0$ 
for any $n$, the theorem shows the nontrivial fact that the limit 
$\lim_{n\rightarrow\infty}\Prob[U>m\varepsilon]$ is continuous 
in $\delta$ at $\delta=0$. Furthermore, the proof provides the 
quantitative estimate 
\begin{equation*}
\Prob\left[U\geq m\varepsilon\right]
\leq\frac{{\mathcal O}\left(\delta^{\frac{1}{2}}\right)}
{{\mathcal O}(\varepsilon)+{\mathcal
O}\left(\delta^{\frac{1}{2}}\right)}+{\mathcal O}
\left(\e^{-n\delta}\right).
\end{equation*}

Theorem \ref{maintheorem} is also very interesting in the 
context of the 
Chv\`atal-Sankoff conjecture which concerns the order of magnitude of 
the fluctuation of the LCS of two random texts.

We recall the assumptions under which we prove our main result and which
will remain valid throughout the rest of this paper: Let $n\in\mathbb{N}$ 
and let $0<\delta<1$ be a fixed constant not 
depending on $n$. We set $m=\lfloor n-\delta n\rfloor$ and define two 
independent random sequences $X=X_1\dots X_m$ and $Y=Y_1\dots Y_n$ by choosing 
$X_1,\dots,X_m$ and $Y_1,\dots,Y_n$ as i.i.d.\ Bernoulli variables with 
parameter $1/2$ (i.e., coin tossing experiments). We then consider alignments 
\begin{equation*}
\begin{array}{ccccccccccc}
\sqcup&\dots&\sqcup&X_1&\sqcup&\dots&\sqcup&X_m&\sqcup&\dots&\sqcup\\
Y_1&\dots&Y_{\xi(1)-1}&Y_{\xi(1)}&Y_{\xi(1)+1}&\dots&Y_{\xi(m)-1}&
Y_{\xi(m)}&Y_{\xi(m)+1}&\dots&Y_n
\end{array}
\end{equation*}
with gaps in $X$ only and attribute to it the score 
\begin{equation*}
S(X,Y;\xi)=\#\{i\in\NN_m:\,X_i=Y_{\xi(i)}\},
\end{equation*}
where $\#$ denotes the cardinality of a set and $\NN_n:=\{1,\dots,n\}$. 
The set of alignments $\xi$ that maximize $S(X,Y;\xi)$ is denoted by 
${\mathcal OA}_{X,Y}$. Of course, this is a random set of alignments, 
since $X$ and $Y$ are random. We write 
$S^*(x,y):=\max_{\xi}S(X,Y;\xi)$ for the maximum score and 
\begin{equation*}
U:=\#\{i\in\NN_m:\,\exists\,\xi,\lambda\in{\mathcal OA}_{X,Y}
\text{ s.t. }\xi(i)\neq\lambda(i)\}
\end{equation*}
for the number of positions where $X$ is not uniquely aligned with $Y$ 
among the alignments with maximum score. $U$ is a random 
variable. 

Theorem \ref{maintheorem} states that 
$\Prob\left[U\geq m\varepsilon\right]<\varepsilon$ 
for all $n$ large enough and $\delta$ small enough. In other words, 
for large $n$ and small gap proportion $\delta$ the optimal alignment 
is locally unique in a $(1-\epsilon)$-proportion of the sequence $X$ 
with probability greater than $1-\epsilon$.  This result is qualitatively representative 
for what occurs 
with regards to the local uniqueness of optimal alignments of random sequences
under {\em arbitrary} scoring functions $S(X,Y)$ as defined in \eqref{scoring 
function} whenever gaps are strongly penalized, i.e., $s(a,\sqcup)$ is a 
negative number of not too small a modulus. Indeed, strong gap penalization 
prevents optimal alignments from having more than a small proportion of 
gaps. Furthermore, allowing gaps only in one of the sequences is merely a 
technical assumption that vastly simplifies the analysis.

Let us now briefly explain the main idea behind the proof of Theorem
\ref{maintheorem}. We define a measure preserving map by picking an 
entry of $X$ at random and flipping it to the ``opposite'' value, i.e., a $0$ is 
changed to a $1$ or vice versa (we imagine $X$ as a line-up of randomly 
tossed fair coins). We denote the sequence obtained in this fashion 
by $\tilde{X}$. Since this operation is measure preserving, we have 
\begin{equation}\label{fundamental}
\E[\Delta]=0,
\end{equation}
where $\Delta:=S^*(\tilde{X},Y)-S^*(X,Y)$. A crucial observation is now 
that when the optimal alignment is nonunique in a large proportion of places, 
then the optimal score tends to increase under this measure-preserving 
map. We illustrate this phenomenon in Example \ref{henry's example1} below. 
Together with \ref{fundamental} this observation implies that the probability 
that the optimal alignment is nonunique in many points is small.

\begin{example}\label{henry's example1}
Consider the case where $n=8$, $m=6$, $\delta=1/4$ and $X$ and $Y$ take the 
values $x=001110$ and $y=11110011$. There are two optimal alignments, $\xi$ 
given by $\xi(i)=i$ for $(i=1,\dots,6)$ or 
\begin{equation*}
\begin{array}{cccccccc}
0&0 &1 &1&1&0&\sqcup&\sqcup\\
1&1 &1 &1&0&0&1&1,
\end{array}
\end{equation*}
and $\lambda$ given by $\lambda(i)=i$ for $(i=1,\dots,4)$ and  
$\lambda(5)=7$, $\lambda(6)=8$ or 
\begin{equation*}
\begin{array}{cccccccc}
0&0 &1 &1&\sqcup&\sqcup&1&0\\
1&1 &1 &1&0&0&1&1.
\end{array}
\end{equation*}
The optimal score is $S^*(x,y)=S(x,y;\xi)=S(x,y;\lambda)=3$. 
\end{example}

The following combinatorial property holds for arbitrary alignments 
$\xi,\lambda$ of $x$ and $y$, not only optimal ones: If $i\in\{1,\dots,m\}$ 
is such that
\begin{equation}\label{bitneq}
y_{\xi(i)}\neq y_{\lambda(i)}
\end{equation}
then flipping $x_i$ to the opposite value 
increases at least one of the scores $S(x,y;\xi)$, 
$S(x,y;\lambda)$ by one unit. In particular, if $\xi$ and $\lambda$ are both 
optimal alignments and condition \eqref{bitneq} holds, then flipping $x_i$ 
to the opposite value increases the optimal score by one unit.
For the chosen values of $x$ and $y$ we find  that $i=5,6$ satisfy condition 
\eqref{bitneq}. Flipping the value of $x_5$ from $1$ to $0$, we find that 
the score of $\xi$ increases to $4$ and the score of $\lambda$ decreases to 
$2$. The maximum score is now $4$. Similarly, flipping $x_6$ from $0$ to $1$, 
the score of $\xi$ decreases to $2$ whereas the score of $\lambda$ increases to 
$4$. The new maximum score is again $4$. If an entry $x_T$ of $x$ is 
flipped, where $T$ is a random index in $\NN_m$, then this implies 
\begin{equation}\label{petit0}
\E[\Delta\|X=x,Y=y,\xi(T)\neq\lambda(T),Y_{\xi(T)}\neq Y_{\lambda(T)}]=1.
\end{equation}

On the other hand, if one of the entries $x_1,\dots,x_4$ that do not satisfy 
\eqref{bitneq} is flipped, then the maximum 
score can either increase or decrease: For $i=1\dots,4$, we have $\xi(i)=
\lambda(i)$. The entries $x_1$ and $x_2$ are aligned with non-matching symbols, 
so that flipping one of these entries increases the optimal score by one. The 
entries $x_3$ and $x_4$ are aligned with matching symbols. In the present 
example, switching one of these entries results in a decrease of the optimal 
score by one unit, though in other cases the maximum score can remain unchanged 
(but it will then be attained by a different alignment). 
Thus, we find that if a random entry $x_{T}$ of $x$ is flipped (where $T\in\NN_m$) 
then for the above choices of $x$ and $y$, 
\begin{equation}\label{petit}
\E[\Delta\|X=x,Y=y,\xi(T)=\lambda(T)]\geq 0.
\end{equation}
For the same reason, if there were any indices $i$ such that $\xi(i)\neq\lambda(i)$ 
where \eqref{bitneq} does not hold, then we would find 
\begin{equation*}
E[\Delta\|X=x,Y=y,\xi(T)\neq\lambda(T),Y_{\xi(T)}=Y_{\lambda(T)}]\geq 0.
\end{equation*}
In conjuction with \eqref{petit0} this implies 
\begin{multline}\label{Delta1}
\E[\Delta\|X=x,Y=y,\xi(T)\neq\lambda(T)]\\
\geq\Prob[Y_{\xi(T)}\neq Y_{\lambda(T)}\|
X=x,Y=y,\xi(T)\neq\lambda(T)].
\end{multline}

The proof of Theorem \ref{maintheorem} exploits a generalization of these observations: 
Lemma \ref{lemma1.2} of Section \ref{alignments} shows that there exist 
two optimal 
alignments $\xi$ and $\lambda$ that differ from each other exactly in 
positions $i$ where the optimal alignment of $X$ and $Y$ is not locally unique. 
In Section \ref{deviations} we show that up to negatively exponentially small 
probability in $n$ the following are true,
\begin{enumerate}
\item[i)] approximately half the points $i\in \mathbb{N}_m$
with $\xi(i)=\lambda(i)$ satisfy $X_i\neq Y_{\lambda(i)}$,
\item[ii)] approximately half the points $i\in \mathbb{N}_m$
with $\xi(i)\neq\lambda(i)$ satisfy $Y_{\xi(i)}\neq Y_{\lambda(i)}$,
\item[iii)] approximately a quarter of points $i\in\mathbb{N}_m$
with $\xi(i)\neq\lambda(i)$ satisfy $X_i\neq Y_{\xi(i)}=Y_{\lambda(i)}$,
\item[iv)] approximately a quarter of points $i\in\mathbb{N}_m$
with $\xi(i)\neq\lambda(i)$ satisfy $X_i=Y_{\xi(i)}=Y_{\lambda(i)}$,
\item[v)] for all $e_1,e_2\in\{0,1\}$ approximately a quarter of points
$i\in\mathbb{N}_m$ satisfy $X_i=e_1$ and $Y_{\xi(i)}=e_2$.
\end{enumerate}

Let $T$ be the random index in $\NN_m$ that corresponds to the entry of $X$ 
that is flipped. By the observations of Example \ref{henry's 
example1}, i)--iv) lead to the following generalization of \eqref{Delta1},
\begin{equation}\label{Delta1'}
\E[\Delta\|X=x,Y=y,\xi(T)\neq\lambda(T)]\geq\frac{1}{2}.
\end{equation}
Likewise, v) leads to the following generalization of \eqref{petit}, 
\begin{equation}\label{petit'}
\E[\Delta\|X=x,Y=y,\xi(T)=\lambda(T)]\geq 0.
\end{equation}
Of course, \eqref{Delta1'} and \eqref{petit'} hold only approximately. 
Much of the work of Section \ref{deviations} 
is devoted to overcoming these imprecisions. For now, let us work with 
the simplified assumption that \eqref{Delta1'} and \eqref{petit'} hold true 
except on a set $F^c$ of pairs $(X,Y)$ with negatively exponentially small 
probability $\Prob[F^c]=\exp(-{\mathcal O}(n))$. Equations  
\eqref{fundamental}, \eqref{Delta1'} and \eqref{petit'} then imply 
\begin{equation*}
0=\E[\Delta]\geq\frac{1}{2}\times\Prob[\xi(T)\neq\lambda(T)]+0\times\Prob[\xi(T)
=\lambda(T)]-1\times\Prob[F^c],
\end{equation*}
so that 
\begin{equation}\label{too tight}
\Prob[\xi(T)\neq\lambda(T)]\leq\exp(-{\mathcal O}(n)).
\end{equation}
When the approximate statements \eqref{Delta1'} and \eqref{petit'} are replaced 
with correct inequalities, \eqref{too tight} turns into the 
weaker claim of Theorem \ref{maintheorem}. 

The structure of the remaining sections of this paper is as follows. 
Section \ref{alignments} serves to introduce the main notation 
relevant to scoring functions, alignments and local uniqueness of 
alignments. We also discuss illustrative examples and prove two 
technical results of preliminary nature. In Section \ref{deviations} 
we introduce events defined in terms of certain empirical distributions 
and their large deviations. These events allow putting the above-made 
approximate statements i)--v) onto a rigorous footing. In Section 
\ref{ergodic} we define formally the measure-preserving map which 
flips a random entry of $X$ to its opposite value. In Lemma \ref{lem3.1} 
of that section we prove that the locations where the optimal alignment
is nonunique tend to introduce a positive bias into $\E[\Delta]$. 
Section \ref{main} finally brings all the elements together in the proof 
of Theorem \ref{maintheorem}.

\section{Alignments and Scores}
\label{alignments}
Let $(x,y)\in\{0,1\}^m\times\{0,1\}^n$ be a pair of sequences of lengths 
$m<n$ over the binary alphabet. Let us consider alignments 
\begin{equation*}
\begin{array}{ccccccccccc}
y_1&\dots&y_{\xi(1)-1}&y_{\xi(1)}&y_{\xi(1)+1}&\dots&y_{\xi(m)-1}&
y_{\xi(m)}&y_{\xi(m)+1}&\dots&y_n\\
\sqcup&\dots&\sqcup&x_1&\sqcup&\dots&\sqcup&x_m&\sqcup&\dots&\sqcup\\
\end{array}
\end{equation*}
of $x$ and $y$ that have $\delta n$ gaps in $x$, where
$m=\lfloor(1-\delta)n\rfloor$. In the above display we marked gaps with the symbol 
$\sqcup$. We identify the set of such 
alignments with the set ${\mathcal A}_{m,n}$ of order-preserving 
injections of $\NN_m$ into $\NN_n:=\{1,\dots,n\}$, that is, $\xi\in
{\mathcal A}_{m,n}$ if and only if $\xi:\NN_m\hookrightarrow\NN_n$ and 
$i<j$ implies $\xi(i)<\xi(j)$.

\begin{lemma}\label{lem1.1}
If $\delta<5/6$, then $\#{\mathcal A}_{m,n}\leq\exp\bigl(n H(\delta)\bigr)$, 
where $H(\delta)=-\bigl(\delta\ln\delta+(1-\delta)\ln(1-\delta)
\bigr)$ is the entropy function.
\end{lemma}

\begin{proof}
Robbins' inequality \cite{Robbins} says that 
\begin{equation*}
\sqrt{2\pi} n^{-n+\frac{1}{12+1}}\leq n!\leq\sqrt{2\pi} 
n^{n+\frac{1}{2}}\e^{-n+\frac{1}{12n}}. 
\end{equation*}
Therefore, 
\begin{multline*}
\#{\mathcal A}_{m,n}=\binom{n}{n(1-\delta)}\\
\leq\frac{\sqrt{2\pi}n^{n+\frac{1}{2}}\e^{-n+\frac{1}{12n}}}
{\sqrt{2\pi}\bigl(n(1-\delta)\bigr)^{n(1-\delta)+\frac{1}{2}}
\e^{-n(1-\delta)+\frac{1}{12n(1-\delta)+1}}\times
\sqrt{2\pi}(n\delta)^{n\delta+\frac{1}{2}}\e^{-n\delta+
\frac{1}{12n\delta+1}}}\\
=\e^{n H(\delta)}\times\frac{\exp\left(\frac{1}{12n}-
\frac{1}{12n(1-\delta)+1}-\frac{1}{12n\delta+1}\right)}
{\sqrt{2\pi(1-\delta)(n-m)}}.
\end{multline*}
Note that the second factor converges to zero for fixed $\delta$ and 
$n\rightarrow\infty$. Moreover, the numerator is $<1$ and for 
$\delta<5/6$ the denominator is $>1$ since $2\pi(1-\delta)(n-m)>
6(1-\delta)>1$.
\end{proof}

We define a scoring function $\{0,1\}^m\times\{0,1\}^n\times 
{\mathcal A}_{m,n}\rightarrow\NN_0$ as follows, 
\begin{equation*}
S(x,y;\xi):=\sum_{i=1}^m s(x_i,y_{\xi(i)}),
\end{equation*}
where $s(0,0)=s(1,1)=1$ and $s(0,1)=s(1,0)=0$. The set of optimal 
alignments of $(x,y)$ is the set of alignments with maximum score, 
\begin{equation*}
{\mathcal OA}_{x,y}:=\left\{\xi\in{\mathcal A}_{m,n}:\, 
S(x,y;\xi)\geq S(x,y;\lambda)\;\forall\,\lambda\in{\mathcal A}_{m,n}
\right\}.
\end{equation*}
We write $S^*(x,y):=\max\{S(x,y;\xi):\,\xi\in{\mathcal A}_{m,n}\}$ 
for the maximum score, as before.

For each $i\in\NN_m$ we define the variable 
\begin{equation*}
u_i(x,y):=\begin{cases}1\quad&\text{if }\exists\,\xi,\lambda\in
{\mathcal OA}_{x,y}\text{ s.t. }\xi(i)\neq\lambda(i),\\
0&\text{otherwise}
\end{cases}
\end{equation*}
that indicates when the image of $i$ under the set of optimal 
alignments is nonunique. We say that the optimal alignment is 
{\em locally nonunique} at $i$ if $u_i(x,y)=1$. We write 
\begin{equation*}
u(x,y):=\sum_{i=1}^m u_i(x,y) 
\end{equation*}
for the number of indices where the optimal alignment is locally 
nonunique. 

The sets ${\mathcal OA}_{x,y}$ and $\{i\in\NN_m:\,
u_i(x,y)=1\}$ can be found via dynamic programming: A $m\times n$ 
matrix $\bigl(score(i,j)\bigr)$ is recursively computed, using the 
rules 
\begin{itemize}
\item[r.i)\;] $score(i,j)=-1$ for $i>j$ or $j>i+\delta n$, 
\item[r.ii)\;] $score(1,j)=s(x_1,y_j)$ for $j=1,\dots,1+\delta n$, 
\item[r.iii)\;] $score(i,j)=s(x_i,y_j)+\max\{score(i-1,k):\,k<j\}$ 
for all other $(i,j)$. 
\end{itemize}
Arguing recursively, one immediately verifies 
that 
\begin{equation*}
S^*(x,y)=\max\{score(m,j):\,j\in\NN_n\},
\end{equation*}
and furthermore that $\xi\in{\mathcal OA}_{x,y}$ if and only if the 
following conditions are satisfied:
\begin{itemize}
\item[c.i)\;] $\xi(m)\in\bigl\{j\in\NN_n:\,score(m,j)=
S^*(x,y)\bigr\}$, 
\item[c.ii)\;] $\xi(i-1)\in\bigl\{j<\xi(i):\,score(i-1,j)=
\max_{k<\xi(i)}score(i-1,k)\bigr\}$ for all $i=2,\dots,m$.\\
\end{itemize}

\begin{example}\label{example1}
Let $x=01010101$ and $y=110101101100$. Then the above 
described dynamic programming algorithm generates the matrix of Figure 
\ref{figure1}, 
\begin{figure}[h]
\begin{center}
\includegraphics{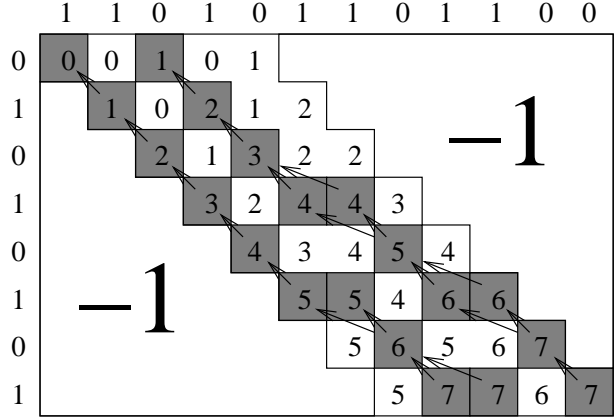}
\end{center}
\caption{The scoring matrix of Example \ref{example1}.}
\label{figure1}
\end{figure}
where it is displayed in tableau format. Optimal paths follow 
the arrows and pass through the shaded entries. The tableau is 
annotated with the generating sequences $x$ and $y$, so that the optimal 
alignments can easily be read off. The following table lists $y$ in the 
top row, followed by a complete list of optimal alignments of $x$ with $y$, 
\begin{equation*}
\begin{array}{cccccccccccc}
{\mathbf 1}&{\mathbf 1}&{\mathbf 0}&{\mathbf 1}&{\mathbf 0}&
{\mathbf 1}&{\mathbf 1}&{\mathbf 0}&{\mathbf 1}&{\mathbf 1}
&{\mathbf 0}&{\mathbf 0}\\
0&1&0&1&0&1&\sqcup&0&1&\sqcup&\sqcup&\sqcup\\
0&1&0&1&0&1&\sqcup&0&\sqcup&1&\sqcup&\sqcup\\
0&1&0&1&0&\sqcup&1&0&1&\sqcup&\sqcup&\sqcup\\
0&1&0&1&0&\sqcup&1&0&\sqcup&1&\sqcup&\sqcup\\
\sqcup&\sqcup&0&1&0&1&\sqcup&0&1&\sqcup&0&1\\
\sqcup&\sqcup&0&1&0&1&\sqcup&0&\sqcup&1&0&1\\
\sqcup&\sqcup&0&1&0&\sqcup&1&0&1&\sqcup&0&1\\
\sqcup&\sqcup&0&1&0&\sqcup&1&0&\sqcup&1&0&1\\
\end{array}
\end{equation*}
Every line of the the tableau in Figure \ref{figure1} 
contains multiple shaded entries. Therefore, $u_i(x,y)=1$ for all 
$i\in\NN_m$ and $u(x,y)=m$.
\end{example}

In the above example we ordered the optimal alignments 
from leftmost to rightmost as located within the tableau, that is, 
alignments are listed in inverse alphabetical order with respect 
to the lateness of gaps. This also provides the idea of proof for the 
following result, which shows that there exist two optimal alignments 
that differ from one another at every point where the optimal alignment 
is locally nonunique.

\begin{lemma}\label{lemma1.2}
For all $(x,y)\in\{0,1\}^m\times\{0,1\}^n$ there exist $\xi,\lambda\in
{\mathcal OA}_{x,y}$ such that $\xi(i)\neq\lambda(i)$ for exactly 
those indices $i\in\NN_m$ for which $u_i(x,y)=1$.
\end{lemma}

\begin{proof}
The claim is clearly true if we can prove that $\xi,\lambda\in
{\mathcal OA}_{x,y}$, where 
\begin{align*}
\xi(i)&:=\min\left\{\psi(i):\,\psi\in{\mathcal OA}_{x,y}\right\}
\quad\forall\,i\in\NN_m,\\
\lambda(i)&:=\max\left\{\psi(i):\,\psi\in{\mathcal OA}_{x,y}\right\}
\quad\forall\,i\in\NN_m.
\end{align*}
If $\lambda\notin{\mathcal OA}_{x,y}$ then there exists an 
index $i\in\NN_m\setminus\{1\}$ such that 
$\widehat{\psi}(i-1)<\lambda(i-1)$ for all 
$\widehat{\psi}\in{\mathcal OA}_{x,y}$ such that $\widehat{\psi}
(i)=\lambda(i)$. On the other hand, there exists $\widehat{\lambda}\in
{\mathcal OA}_{x,y}$ such that $\widehat{\lambda}(i-1)=\lambda(i-1)$, 
and therefore it is necessarily the case that $\widehat{\lambda}(i)
<\widehat{\psi}(i)=\lambda(i)$. But $\widehat{\lambda}$ and 
$\widehat{\psi}$ satisfy condition c.ii), that is, 
\begin{align*}
\widehat{\lambda}(i-1)&\in\Bigl\{j<\widehat{\lambda}(i):\,score(i-1,j)
=\max_{k<\widehat{\lambda}(i)}score(i-1,k)\Bigr\},\\
\widehat{\psi}(i-1)&\in\Bigl\{j<\widehat{\psi}(i):\,score(i-1,j)
=\max_{k<\widehat{\psi}(i)}score(i-1,k)\Bigr\}.
\end{align*}
Therefore, either $\max_{k<\widehat{\lambda}(i)}score(i-1,k)
=\max_{k<\widehat{\psi}(i)}score(i-1,k)$ and then there exists 
$\psi\in{\mathcal OA}_{x,y}$ such that $\psi(i)=\widehat{\psi}(i)=
\lambda(i)$ and $\psi(i-1)=\widehat{\lambda}(i-1)=\lambda(i-1)$, or else 
$\max_{k<\widehat{\lambda}(i)}score(i-1,k)<\max_{k<\widehat{\psi}
(i)}score(i-1,k)$ and then $\widehat{\psi}(i-1)>\widehat{\lambda}
(i-1)=\lambda(i-1)$. In either case we have a contradiction, and this 
shows that $\lambda\in{\mathcal OA}_{x,y}$ indeed. The proof that 
$\xi\in{\mathcal OA}_{x,y}$ is analogous. 
\end{proof}

Note that the alignments $\xi$ and $\lambda$ constructed in the proof 
of Lemma \ref{lemma1.2} are uniquely determined by $(x,y)$. Furthermore, 
they satisfy the relation $\xi\leq\lambda$ which is defined by $\xi(i)\leq 
\lambda(i)$ for all $i\in\NN_m$.

\section{Large Deviations of Some Empirical Distributions}
\label{deviations}
In this section we establish a rigorous version of the approximate 
inequalities \eqref{Delta1'},\eqref{petit'} and statements 
i)--v) of Section \ref{problem setting}. 
Recall that $X=X_1\dots X_m$ and $Y=Y_1\dots Y_n$ are two independent 
random sequences that consist of i.i.d.\ 
standard Bernoulli variables $X_i,Y_j\sim{\mathscr B}(1/2)$ 
defined on some probability space $(\Omega,{\mathscr F},\Prob)$. 
We write $U_i$ and $U$ for the random variables $u_i(X,Y)$ and 
$u(X,Y)$ respectively. Furthermore, we think of $\delta\in(0,1)$ 
as a fixed gap proportion that relates $m$ to $n$ via 
$m=\lfloor(1-\delta)n\rfloor$. 

For $\xi\in{\mathcal A}_{m,n}$ fixed and $\omega\in\Omega$ let 
$\widehat{{\mathscr D}}_{\xi}(\omega)$ be the empirical distribution 
of $(X_i(\omega),Y_{\xi(i)}(\omega))$ over $i\in\NN_m$, 
i.e., the distribution of $(X_T(\omega),Y_{\xi(T)}(\omega))$ 
when $T\sim{\mathscr U}(\NN_m)$ is a random index with uniform 
distribution on $\NN_m$. Yet another way to define this distribution 
is to require that for all $(e_1,e_2)\in\{0,1\}^2$, 
\begin{equation*}
\Prob_{\widehat{{\mathscr D}}_{\xi}(\omega)}\bigl[(e_1,e_2)\bigr]
=\frac{1}{m}\times\#\Bigl\{i\in\NN_m:\,X_i(\omega)=e_1,\,
Y_{\xi(i)}(\omega)=e_2\Bigr\}.
\end{equation*}

\begin{example}\label{henry's example2}
Let $x,y$ and $\xi$ be chosen as in Example \ref{henry's example1}. 
If $\omega\in\Omega$ is chosen such that $X(\omega)=x$ and 
$Y(\omega)=y$ then 
\begin{equation*}
\Prob_{\widehat{{\mathscr D}}_{\xi}(\omega)}\bigl[(0,0)\bigr]
=\frac{1}{6},\;
\Prob_{\widehat{{\mathscr D}}_{\xi}(\omega)}\bigl[(0,1)\bigr]
=\frac{2}{6},\;
\Prob_{\widehat{{\mathscr D}}_{\xi}(\omega)}\bigl[(1,0)\bigr]
=\frac{1}{6},\;
\Prob_{\widehat{{\mathscr D}}_{\xi}(\omega)}\bigl[(1,1)\bigr]
=\frac{2}{6}.
\end{equation*}
Note that $\widehat{{\mathscr D}}_{\xi}$ only depends on $x$ and $y$.
\end{example}

Let $E_{\xi}$ be the event that 
\begin{equation}\label{round10.1}
\max_{(e_1,e_2)\in\{0,1\}^2}\left|\Prob_{\widehat{{\mathscr D}}_{\xi}
(\omega)}\bigl[(e_1,e_2)\bigr]-1/4\right|<
\sqrt{\frac{9H(\delta)}{4(1-\delta)}}=:\epsilon_1(\delta),
\end{equation}
in other words,
\begin{equation*}
E_{\xi}:=\left\{\omega\in\Omega:\,\left\|\widehat{{\mathscr D}}_{\xi}
(\omega)-{\mathscr B}(1/2)\otimes{\mathscr B}(1/2)\right\|<\epsilon_1
(\delta)\right\}.
\end{equation*}
Let us furthermore define the event 
\begin{equation*}
E_{m,n}:=\bigcap_{\xi\in{\mathcal A}_{m,n}}E_{\xi}.
\end{equation*}
In a similar vein we define empirical distributions and events 
relating to $(X_i,Y_{\xi(i)},Y_{\lambda(i)})$ as follows: Let 
$\varepsilon>0$ and $\xi,\lambda\in{\mathcal A}_{m,n}$ 
be fixed such that $\xi\leq\lambda$ and $d(\xi,\lambda):=\#\{i:\,\xi(i)
\neq\lambda(i)\}\geq m\varepsilon$. For all $i$ such that 
$\xi(i)=\lambda(i)$ we define the random variables 
\begin{equation*}
R^{\xi,\lambda}_i:=\begin{cases}1\quad&\text{if }X_i\neq Y_{\xi(i)}\\
-1&\text{if }X_i=Y_{\xi(i)}.
\end{cases}
\end{equation*}
Likewise, for all $i$ such that $\xi(i)\neq
\lambda(i)$ we define the random variables 
\begin{equation*}
R^{\xi,\lambda}_i:=\begin{cases}0\quad&\text{if }Y_{\xi(i)}\neq 
Y_{\lambda(i)},\\
1\quad&\text{if }X_i\neq Y_{\xi(i)}=Y_{\lambda(i)},\\
-1\quad&\text{if }X_i=Y_{\xi(i)}=Y_{\lambda(i)}.
\end{cases}
\end{equation*}
We now introduce the following notation:\\

\begin{itemize}
\item $\widehat{{\mathscr L}}^{agree}_{\xi,\lambda}(\omega)$,  
$\widehat{{\mathscr L}}^{disag}_{\xi,\lambda}(\omega)$ and 
$\widehat{{\mathscr L}}^{unif}_{\xi,\lambda}(\omega)$ are the 
empirical distributions of $R^{\xi,\lambda}_{i}(\omega)$ over 
$\{i\in\NN_m:\,\xi(i)=\lambda(i)\}$, $\{i\in\NN_m:\,\xi(i)\neq\lambda(i)\}$ 
and $i\in\NN_m$ respectively, 
\item ${\mathscr J}^{agree}$ is the distribution on 
$\{-1,1\}$ defined by $\Prob_{{\mathscr J}^{agree}}[-1]=1/2=
\Prob_{{\mathscr J}^{agree}}[1]$, and 
${\mathscr J}^{disag}={\mathscr J}^{unif}$ the distribution 
on $\{-1,0,1\}$ defined by 
$\Prob_{{\mathscr J}^{disag}}[-1]=1/4=\Prob_{{\mathscr J}^{disag}}[1]$ 
and $\Prob_{{\mathscr J}^{disag}}[0]=1/2$, 
\item parameters $\epsilon_2$, $\epsilon_3$ and $\epsilon_4$ are 
determined in terms of $\delta$ and $\varepsilon$ via  the relations 
\begin{align*}
\epsilon_2(\delta,\varepsilon)&:=\sqrt{\frac{3H(\delta)}{2(1-\delta)
\varepsilon}},
&\epsilon_3(\delta,\varepsilon)&:=\sqrt{\frac{27H(\delta)}{8(1-\delta)
\varepsilon}},
&\epsilon_4(\delta,\varepsilon)&:=\sqrt{\frac{3H(\delta)}{2(1-\delta)
(1-\varepsilon)}}.
\end{align*}
\end{itemize}
With this notation, we define the following events, 
\begin{align*}
F_{\xi,\lambda}^{\varepsilon}&:=\left\{\omega\in\Omega:\,
\left\|\widehat{{\mathscr L}}^{agree}_{\xi,\lambda}-{\mathscr J}^{agree}
\right\|<\epsilon_2(\delta,\varepsilon)\right\},\\
G_{\xi,\lambda}^{\varepsilon}&:=\left\{\omega\in\Omega:\,
\left\|\widehat{{\mathscr L}}^{disag}_{\xi,\lambda}-{\mathscr J}^{disag}
\right\|<\epsilon_3(\delta,\varepsilon)\right\},\\
H_{\xi,\lambda}^{\varepsilon}&:=\left\{\omega\in\Omega:\,
\left\|\widehat{{\mathscr L}}^{unif}_{\xi,\lambda}-{\mathscr J}^{unif}
\right\|<\epsilon_4(\delta,\varepsilon)+2\varepsilon\right\},\\
F_{m,n,\varepsilon}&:=\bigcap_{\{(\xi,\lambda):\xi\leq\lambda, 
m\varepsilon\leq d(\xi,\lambda)\leq m(1-\varepsilon)\}}\hspace{-1.7cm}
\left(F_{\xi,\lambda}^{\varepsilon}
\cap G_{\xi,\lambda}^{\varepsilon}\right)\quad\cap
\bigcap_{\{(\xi,\lambda): \xi\leq\lambda, m(1-\varepsilon)<d(\xi,\lambda)\}}
\hspace{-.8cm}
H_{\xi,\lambda}^{\varepsilon}.
\end{align*} 

\begin{lemma}\label{lem2.1} For all $\delta<5/6$ and $\varepsilon>0$,  
\begin{itemize}
\item[i)\;] $\Prob\bigl[E_{m,n}^c\bigr]\leq 8\e^{-nH(\delta)}$,
\item[ii)\;] $\Prob\bigl[F_{m,n,\varepsilon}^c\bigr]\leq 10\e^{-nH(
\delta)}$.
\end{itemize}
\end{lemma}

\begin{proof}
The Azuma-Hoeffding Theorem \cite{azuma,hoeffding} says that 
if $(V_0,\dots,V_m)$ is a martingale with $V_0\equiv 0$ and 
$\Prob[|V_k-V_{k-1}|\leq a]=1$ for all $k\in\NN_m$ then 
\begin{equation*}
\Prob[V_m\geq mb]\leq\exp\left(-\frac{m b^2}{2 a^2}\right)
\end{equation*} 
for all $b>0$. For $\xi\in{\mathcal A}_{m,n}$ and $(e_1,e_2)
\in\{0,1\}^2$ fixed let 
\begin{equation*}
Z_i(\omega):=\begin{cases}1\quad&\text{if }\bigl(X_{i}(\omega),
Y_{\xi(i)}(\omega)\bigr)=(e_1,e_2),\\
0\quad&\text{otherwise}.
\end{cases}
\end{equation*}
Then $Z_i$ $(i\in\NN_m$ are i.i.d.\ random variables with expectation 
$\E[Z_i]=1/4$. If we set $V_0:=0$ and 
\begin{equation*}
V_k:=\sum_{i=1}^k\bigl(Z_i-\E[Z_i]\bigr)\qquad(k\in\NN_m),
\end{equation*}
then $(V_0,\dots,V_m)$ is a martingale with $|V_k-V_{k-1}|\leq 
3/4$ for all $k$. By the Azuma-Hoeffding Theorem, we have 
\begin{equation*}
\Prob\left[\frac{1}{m}\sum_{i=1}^m Z_i-\frac{1}{4}
\geq\epsilon_1(\delta)\right]
=\Prob\left[\frac{1}{m}V_m\geq\epsilon_1(\delta)\right]
\leq\exp\left(-\frac{8m\epsilon_1^2
(\delta)}{9}\right)\leq\e^{-2nH(\delta)}.
\end{equation*}
Applying the same reasoning to the martingale $(-V_0,\dots,
-V_m)$, we find 
\begin{equation*}
\Prob\left[\frac{1}{m}\sum_{i=1}^m Z_i-\frac{1}{4}
\leq -\epsilon_1(\delta)\right]\leq\e^{-2nH(\delta)},
\end{equation*}
so that 
\begin{equation*}
\Prob\left[\left|
\Prob_{\widehat{{\mathscr D}}_{\xi}}\bigl[(e_1,e_2)\bigr]
-\frac{1}{4}\right|\geq\epsilon_1(\delta)\right]\leq 2\e^{-2nH(\delta)}.
\end{equation*}
Since this is true for all $(e_1,e_2)\in\{0,1\}^2$, simple 
union bounds show that 
\begin{equation*}
\Prob\left[E_{\xi}^c\right]\leq 8\e^{-2nH(\delta)}
\end{equation*}
and 
\begin{equation*}
\Prob\left[E_{m,n}^c\right]=\Prob\left[\bigcup_{\xi\in{\mathcal 
A}_{m,n}}E_{\xi}^c\right]\leq\#{\mathcal A}_{m,n}\times 
8\e^{-2nH(\delta)}\stackrel{Lem \ref{lem1.1}}{\leq}
8\e^{-nH(\delta)}.
\end{equation*}

ii)\; The proof of the second part is similar: Let $(\xi,\lambda)$ 
be such that $\xi\leq\lambda$ and $d(\xi,\lambda)\geq m\varepsilon$. For 
all $i$ such that $\xi(i)\neq\lambda(i)$ we have $\xi(i)<\lambda(i)$. 
For $e\in\{-1,0,1\}$ fixed let 
\begin{equation*}
Z_i:=\begin{cases}1\quad&\text{if }R_i^{\xi,\lambda}=e,\\
0\quad&\text{otherwise}
\end{cases}
\qquad(i\in\{k:\,\xi(k)\neq\lambda(k)\}.
\end{equation*}
Then we have 
\begin{equation*}
\E[Z_i]=\begin{cases}\frac{1}{4}\quad&\text{if }e\in\{-1,1\},\\
\frac{1}{2}\quad&\text{if }e=0.,
\end{cases}
\end{equation*}
so that $|Z_i-\E[Z_i]|\leq 3/4$ in all three cases. 
Furthermore, the random variables 
\begin{equation*}
(Z_i-E[Z_i]),\qquad(i\in\{k:\,\xi(k)\neq\lambda(k)\})
\end{equation*}
are i.i.d.\ 
with distribution ${\mathscr J}^{disag}$. This is seen by induction, 
using the observation that for all index sets 
\begin{equation*}
I\subset\{k:\,\xi(k)\neq\lambda(k)\},
\end{equation*}
if $i_{\max}=\max I$ then $X_{i_{\max}}$ and 
$Y_{\lambda(i_{\max})}$ do not appear in any of the expressions 
$(X_i,Y_{\xi(i)}, Y_{\lambda(i)})$ $(i\in I\setminus\{i_{\max}\})$, 
so that independently of the value of $Y_{\xi(i_{\max})}$ 
(which could have appeared in the above expressions at most once 
as $Y_{\lambda(i)}$), we have 
\begin{multline*}
\Prob\left[Y_{\lambda(i_{\max})}\neq Y_{\xi(i_{\max})}\right]
=\frac{1}{2},\quad\Prob\left[X_{i_{\max}}\neq Y_{\xi(i_{\max})}
=Y_{\lambda(i_{\max})}\right]=\frac{1}{4},\\
\quad\text{and}\quad
\Prob\left[X_{i_{\max}}=Y_{\xi(i_{\max})}=Y_{\lambda(i_{\max})}
\right]=\frac{1}{4}.
\end{multline*}
We define $V_0\equiv 0$ and for $k\in\NN_{d(\xi,\lambda)}$, 
$V_k:=V_{k-1}+Z_i-\E[Z_i]$, where 
\begin{equation*}
i=\min\left\{l\in\NN_m:\,\#\left\{j\leq l:\,
\xi(j)\neq\lambda(j)\right\}=k\;\right\}.
\end{equation*}
Then $(V_0,\dots,V_{d(\xi,\lambda)})$ is a martingale, and 
arguing as above by ways of the Azuma-Hoeffding Theorem, we find 
\begin{multline*}
\Prob\left[\left|
\Prob_{\widehat{{\mathscr L}}_{\xi,\lambda}^{disag}}[e]
-\Prob_{{\mathscr J}^{disag}}[e]\right|\geq\epsilon_3(\delta,
\varepsilon)\right]\\
=\Prob\left[\left|\frac{1}{d(\xi,\eta)}V_{d(\xi,\eta)}\right|
\geq\epsilon_3(\delta,\varepsilon)\right]\leq 2\exp
\left(-\frac{8 d(\xi,\lambda)\epsilon_3^2(\delta,\varepsilon)}
{9}\right)\leq2\e^{-3nH(\delta)}.
\end{multline*}
Since this holds for all $e\in\{-1,0,1\}$, we have 
\begin{equation}\label{star.1}
\Prob\left[G_{\xi,\lambda}^c\right]\leq 6\e^{-3nH(\delta)}
\end{equation}
whenever $d(\xi,\lambda)\geq m\varepsilon$ and $\xi\leq\lambda$. 

If it is even the case that
$d(\xi,\lambda)>m(1-\varepsilon)$, then we find 
\begin{equation*}
\Prob\left[\left|
\Prob_{\widehat{{\mathscr L}}_{\xi,\lambda}^{disag}}[e]
-\Prob_{{\mathscr J}^{disag}}[e]\right|\geq\epsilon_4(\delta,\varepsilon)
\right]\leq 2\exp\left(-\frac{8 d(\xi,\lambda)\epsilon_4^2
(\delta,\varepsilon)}{9}\right)\leq2\e^{-3nH(\delta)},
\end{equation*}
and also 
\begin{equation*}
\left|\Prob_{\widehat{{\mathscr L}}_{\xi,\lambda}^{unif}}[e]
-\Prob_{\widehat{{\mathscr L}}_{\xi,\lambda}^{disag}}[e]\right|
\leq 2\varepsilon.
\end{equation*}
Therefore, 
\begin{multline*}
\Prob\left[\left|\Prob_{\widehat{{\mathscr L}}_{\xi,\lambda}^{unif}}[e]
-\Prob_{{\mathscr J}^{unif}}[e]\right|\geq
\epsilon_4(\delta,\varepsilon)+2 \varepsilon\right]\\
\leq
\Prob\left[\left|\Prob_{\widehat{{\mathscr L}}_{\xi,\lambda}^{disag}}[e]
-\Prob_{{\mathscr J}^{disag}}[e]\right|\geq\epsilon_4(\delta,\varepsilon)\right]
\leq 2\e^{-3nH(\delta)}.
\end{multline*}
Since this holds for all $e\in\{-1,0,1\}$, we have 
\begin{equation}\label{star.2}
\Prob\left[H_{\xi,\lambda}^c\right]\leq 6\e^{-3nH(\delta)}
\end{equation}
whenever $d(\xi,\lambda)>m(1-\varepsilon)$ and $\xi\leq\lambda$.

Next, let $\xi\leq\lambda$ be such that $d(\xi,\lambda)\leq
m(1-\varepsilon)$, and for $e\in\{-1,1\}$ fixed let 
\begin{equation*}
Z_i:=\begin{cases}1\quad\text{if }R^{\xi,\lambda}_i=e,\\
0\quad\text{otherwise}
\end{cases}
\qquad(i\in\{k:\,\xi(k)=\lambda(k)\}).
\end{equation*}
Then $\E[Z_i]=1/2$ 
so that $|Z_i-\E[Z_i]|=1/2$, and 
\begin{equation*}
(Z_i-E[Z_i]),\qquad(i\in\{k:\,\xi(k)=\lambda(k)\})
\end{equation*}
are i.i.d.\ random variables with distribution ${\mathscr J}^{agree}$. 
Let $V_0:= 0$, and for $k\in\NN_{m-d(\xi,\lambda)}$, 
$V_k:=V_{k-1}+Z_i-\E[Z_i]$, where 
\begin{equation*}
i=\min\left\{l\in\NN:\,\#\left\{j\leq l:\,\xi(j)=\lambda(j)\right\}
=k\;\right\}.
\end{equation*}
Then $(V_0,\dots,V_{m-d(\xi,\lambda)})$ is a martingale and the 
large deviations argument from above shows that 
\begin{multline*}
\Prob\left[\left|
\Prob_{\widehat{{\mathscr L}}_{\xi,\lambda}^{agree}}[e]
-\Prob_{{\mathscr J}^{agree}}[e]\right|\geq\epsilon_2(\delta,\varepsilon)
\right]\\
=\Prob\left[\left|\frac{1}{m-d(\xi,\lambda)}V_{m-d(\xi,\lambda)}\right|
\geq\epsilon_2(\delta,\varepsilon)\right]
\leq 2\e^{-2(m-d(\xi,\lambda))\epsilon_2^2}
\leq 2\e^{-3nH(\delta)}.
\end{multline*}
Since this holds for both $e\in\{1,-1\}$, we find 
\begin{equation}\label{star.3}
\Prob\left[F_{\xi,\lambda}^c\right]\leq 4\e^{-3nH(\delta)}
\end{equation}
whenever $d(\xi,\lambda)\leq m(1-\varepsilon)$ and $\xi\leq\lambda$. 

Finally, the combination of equations \eqref{star.1}, \eqref{star.2}, 
\eqref{star.3} and Lemma \ref{lem1.1} shows that 
\begin{align*}
\Prob\left[F_{m,n}^c\right]&\leq\sum_{\{(\xi,\lambda):
\xi\leq\lambda, m\varepsilon\leq d(\xi,\lambda)\leq m(1-\varepsilon)\}}
\hspace{-1.8cm}\left(\Prob\left[F_{\xi,\lambda}^c\right]+
\Prob\left[G_{\xi,\lambda}^c\right]\right)+\sum_{\{(\xi,\lambda):\xi\leq
\lambda, m(1-\varepsilon)< d(\xi,\lambda)\}}
\hspace{-1.4cm}\Prob\left[H_{\xi,\lambda}^c\right]\\
&\leq\left({\mathcal A}_{m,n}\right)^2\times\left(10\e^{-3nH(\delta)}
\right)\\
&\leq 10\e^{-nH(\delta)}.
\end{align*}
\end{proof}

\section{An Ergodic Map}
\label{ergodic}
Let us now introduce an ergodic map as follows: Let $T\sim{\mathscr 
U}(\NN_m)$ be a uniform random variable on $\NN_m$. By Kolmogorov's 
theorem we may assume without loss of generality that $(\Omega,
{\mathscr F},\Prob)$ is extended so that $T$ is defined on $\Omega$ 
and independent of the $X_i$ and $Y_j$. Let us define new
random sequences $\tilde{X}=\tilde{X}_1\dots\tilde{X}_m$ and 
$\tilde{Y}=\tilde{Y}_1\dots\tilde{Y}_n$ by setting $\tilde{Y}:=
Y$, $\tilde{X}_i:=X_i$ for all $i\neq T$ and 
\begin{equation*}
\tilde{X}_T:=X_T+1\mod 2.
\end{equation*}
In other words, $(\tilde{X},\tilde{Y})$ is obtained from $(X,Y)$ 
by flipping one random bit of $X$ and keeping all other entries 
of $X$ and $Y$ unchanged. The map $(X,Y)\mapsto(\tilde{X},\tilde{Y})$ 
is measure-preserving, since $\tilde{X}$ and $\tilde{Y}$ again 
consist of i.i.d.\ standard Bernoulli variables. Therefore, 
\begin{equation}\label{3.1}
\E[\Delta]=0
\end{equation}
where $\Delta:=S^*(\tilde{X},\tilde{Y})-S^*(X,Y)$. The construction 
in the proof of Lemma \ref{lemma1.2} shows that there exists a 
$\sigma(X,Y)$-measurable map 
\begin{align*}
(\Xi,\Lambda):\,\Omega&\rightarrow{\mathcal A}_{m,n}\times{\mathcal 
A}_{m,n},\\
\omega&\mapsto\bigl(\Xi_{\omega},\Lambda_{\omega}\bigr)
\end{align*}
such that for all $\omega\in\Omega$, $\Xi_{\omega}\leq
\Lambda_{\omega}$ and 
\begin{equation*}
\left\{i:\,\Xi_{\omega}(i)\neq\Lambda_{\omega}(i)\right\}
=\left\{i:\,U_i(\omega)=1\right\}.
\end{equation*}
Furthermore, $X$ and $Y$ define the $\sigma(X,Y)$-measurable events 
$E_{m,n}$, $F_{m,n}$ and the $\sigma(X,Y)$-measurable random 
variable $U$ introduced in Section \ref{deviations}. The following 
two lemmas show how these objects affect $\Delta$ and will be the 
key tools in the proof of the main theorem of this paper.

\begin{lemma}\label{lem3.1}
For all $\delta<5/6$ and $\varepsilon>0$, 
\begin{multline*}
\E\left[\Delta\,\Bigr\|\, U\geq m\varepsilon, F_{m,n}\right]
\geq
\left[\frac{1}{2}-3\max\left(\epsilon_4(\delta,\varepsilon)
+2\varepsilon
,\epsilon_3(\delta,\varepsilon)\right)\right]\times\varepsilon\\
+\min\left[\frac{1}{2}-3\left(\epsilon_4(\delta,\varepsilon)+\varepsilon
\right),-2\epsilon_2(\delta,\varepsilon)\right]\times(1-\varepsilon).
\end{multline*}
\end{lemma}

\begin{proof}
A key observation is that $Y_{\Xi(T)}\neq Y_{\Lambda(T)}$ 
implies $\Delta=1$: Without loss of generality we may assume 
that $\tilde{X}_T=Y_{\Xi(T)}$, so that 
$X_T\neq Y_{\Xi(T)}$ and $S^*(X,Y)=\sum_{i\neq T}
s(X_i,Y_{\Xi(i)})$. But then we have 
\begin{multline*}
S^*(X,Y)+1\geq S^*(\tilde{X},\tilde{Y})\geq S(\tilde{X},
\tilde{Y};\Xi)\\
=\sum_{i\neq T}s(X_i,Y_{\Xi(i)})+s(\tilde{X}_T,Y_{\Xi(T)})
=S^*(X,Y)+1,
\end{multline*}
so that $\Delta=1$ indeed. Likewise, $X_T\neq Y_{\Xi(T)}
=Y_{\Lambda(T)}$ implies $\Delta=1$. Using these facts and 
the trivial lower bound $\Delta\geq -1$, we have 
\begin{align}
\E&\left[\Delta\,\Bigr\|\, U\geq m(1-\varepsilon),F_{m,n}\right]
\nonumber\\
&\geq\E\left[1\times\Prob_{\widehat{{\mathscr L}}^{unif}_{\Xi,
\Lambda}}[0]+1\times\Prob_{\widehat{{\mathscr L}}^{unif}_{\Xi,
\Lambda}}[1]-1\times\Prob_{\widehat{{\mathscr L}}^{unif}_{\Xi,
\Lambda}}[-1]\,\Bigr\|\, U\geq m(1-\varepsilon), F_{m,n}\right]
\nonumber\\
&\geq 1\times\left(\Prob_{{\mathscr J}^{unif}}[0]-\epsilon_4
(\delta,\varepsilon)-2\varepsilon\right)+1\times\left(\Prob_{{\mathscr J}
^{unif}}[1]-\epsilon_4(\delta,\varepsilon)-2\varepsilon\right)\nonumber\\
&\qquad\qquad-1\times\left(\Prob_{{\mathscr J}^{unif}}[-1]
+\epsilon_4(\delta,\varepsilon)+2\varepsilon\right)
\nonumber\\
&=\frac{1}{2}-3\left(\epsilon_4(\delta,\varepsilon)+2\varepsilon\right),
\label{round1}
\end{align}
\begin{align}
\E&\left[\Delta\,\Bigr\|\,m(1-\varepsilon)\geq U\geq m\varepsilon, 
F_{m,n}, \Xi(T)\neq\Lambda(T)\right]
\nonumber\\
&\geq\E\Bigl[1\times\Prob_{\widehat{{\mathscr L}}^{disag}_{\Xi,
\Lambda}}[0]+1\times\Prob_{\widehat{{\mathscr L}}^{disag}_{\Xi,
\Lambda}}[1]-1\times\Prob_{\widehat{{\mathscr L}}^{disag}_{\Xi,
\Lambda}}[-1]\,\Bigr\|\,m(1-\varepsilon)\geq U\geq m\varepsilon,
\nonumber\\
&\hspace{4cm} F_{m,n},\Xi(T)\neq\Lambda(T)\Bigr]
\nonumber\\
&\geq 1\times\left(\Prob_{{\mathscr J}^{disag}}[0]-\epsilon_3
(\delta,\varepsilon)\right)+1\times\left(\Prob_{{\mathscr J}
^{disag}}[1]-\epsilon_3(\delta,\varepsilon)\right)\nonumber\\
&\qquad\qquad-1\times\left(\Prob_{{\mathscr J}^{disag}}[-1]
+\epsilon_3(\delta,\varepsilon)\right)
\nonumber\\
&=\frac{1}{2}-3\epsilon_3(\delta,\varepsilon),
\label{round2}
\end{align}
\begin{align}
\E&\left[\Delta\,\Bigr\|\,m(1-\varepsilon)\geq U\geq m\varepsilon, 
F_{m,n}, \Xi(T)=\Lambda(T)\right]
\nonumber\\
&\geq\E\Bigl[1\times\Prob_{\widehat{{\mathscr L}}^{agree}_{\Xi,
\Lambda}}[1]-1\times\Prob_{\widehat{{\mathscr L}}^{agree}_{\Xi,
\Lambda}}[-1]\,\Bigr\|\,m(1-\varepsilon)\geq U\geq m\varepsilon, 
F_{m,n},\Xi(T)=\Lambda(T)\Bigr]
\nonumber\\
&\geq 1\times\left(\Prob_{{\mathscr J}^{agree}}[1]-\epsilon_2
(\delta,\varepsilon)\right)-1\times\left(\Prob_{{\mathscr J}^{agree}}[-1]
+\epsilon_2(\delta,\varepsilon)\right)
\nonumber\\
&=-2\epsilon_2(\delta,\varepsilon),
\label{round3}
\end{align}
Putting the pieces together, we find 
\begin{align*}
\E&\left[\Delta\,\Bigr\|\,U\geq m\varepsilon, F_{m,n}\right]\\
&\geq\E\left[\Delta\,\Bigr\|\,U\geq m(1-\varepsilon), F_{m,n}\right]
\times\Bigl(\Prob\left[U\geq m(1-\varepsilon), \Xi(T)\neq
\Lambda(T)\,\Bigr\|\,U\geq m\varepsilon, F_{m,n}\right]\\
&\hspace{4cm}+\Prob\left[U\geq m(1-\varepsilon),\Xi(T)=\Lambda
(T)\,\Bigr\|\,U\geq m\varepsilon, F_{m,n}\right]\Bigr)\\
&+\E\left[\Delta\,\Bigr\|\,m(1-\varepsilon)\geq U\geq m\varepsilon, 
F_{m,n},\Xi(T)\neq\Lambda(T)\right]\\
&\hspace{4cm}\times
\Prob\left[m(1-\varepsilon)\geq U\geq m\varepsilon, 
\Xi(T)\neq\Lambda(T)\,\Bigr\|\,U\geq m\varepsilon, F_{m,n}\right]\\
&+\E\left[\Delta\,\Bigr\|\,m(1-\varepsilon)\geq U\geq m\varepsilon, 
F_{m,n},\Xi(T)=\Lambda(T)\right]\\
&\hspace{4cm}\times
\Prob\left[m(1-\varepsilon)\geq U\geq m\varepsilon, 
\Xi(T)=\Lambda(T)\,\Bigr\|\,U\geq m\varepsilon, F_{m,n}\right]
\end{align*}
\begin{align*}
&\stackrel{\eqref{round1},\eqref{round2},\eqref{round3}}{\geq}
\left[\frac{1}{2}-3\max\left(\epsilon_4(\delta,\varepsilon)+2\varepsilon,
\epsilon_3(\delta,\varepsilon)\right)\right]\times
\Prob\left[\Xi(T)\neq\Lambda(T)\,\Bigr\|\,U\geq m\varepsilon, 
F_{m,n}\right]\\
&\hspace{2cm}+\min\left[\frac{1}{2}-3\left(\epsilon_4(\delta,
\varepsilon)+2\varepsilon\right),
-2\epsilon_2(\delta,\varepsilon)\right]\times\Prob\left[\Xi(T)=\Lambda(T)\,
\Bigr\|\,U\geq m\varepsilon, F_{m,n}\right]\\
&\hspace{.7cm}\geq\left[\frac{1}{2}-3\max\left(\epsilon_4(\delta,
\varepsilon)+2\varepsilon,
\epsilon_3(\delta,\varepsilon)\right)\right]\times\varepsilon+
\min\left[\frac{1}{2}-3\left(\epsilon_4(\delta,\varepsilon)+
2\varepsilon\right),
-2\epsilon_2(\delta,\varepsilon)\right]\times(1-\varepsilon).
\end{align*}
\end{proof}

\begin{lemma}\label{lem3.2}
For any $\sigma(X,Y)$-measurable event $B$ and all $\delta<5/6$, 
we have 
\begin{equation*}
\int_B\Delta d\Prob\geq -4\epsilon_1(\delta)-8\e^{-nH(\delta)}.
\end{equation*}
\end{lemma}

\begin{proof}
\begin{equation}
\int_B\Delta d\Prob\geq\int_{B\cap E_{m,n}}\Delta d\Prob-
\Prob[E_{m,n}^c]
\stackrel{Lem\ref{lem2.1}}{\geq}\int_{B\cap E_{m,n}}
\Delta d\Prob-8\e^{-n H(\delta)}.\label{3.2}
\end{equation}
Clearly, for all $\omega\in\Omega$, 
\begin{equation*}
\Delta(\omega)\geq S\bigl(\tilde{X}(\omega),\tilde{Y}(\omega);
\Xi_{\omega}\bigr)-S\bigl(X(\omega),Y(\omega);\Xi_{\omega}\bigr).
\end{equation*}
Therefore, 
\begin{align*}
\int_{B\cap E_{m,n}}\Delta d\Prob&\geq
\int_{B\cap E_{m,n}}\left[S\left(\tilde{X},\tilde{Y};\Xi\right)
-S\left(X,Y;\Xi\right)\right]d\Prob\\
&=\int_{B\cap E_{m,n}}\left[s\left(\tilde{X}_T,Y_{\Xi(T)}\right)-
s\left(X_T,Y_{\Xi(T)}\right)\right]d\Prob\left[\left(X(\omega),Y
(\omega),T(\omega)\right)\right]\\
&=\int_{B\cap E_{m,n}}\E\left[s\left(\tilde{X}_T,Y_{\Xi(T)}\right)-
s\left(X_T,Y_{\Xi(T)}\right)\,\Bigr\|\,(X,Y)\right]
d\Prob\left[\left(X(\omega),Y(\omega)\right)\right]\\
&\stackrel{\eqref{round10.1}}{\geq}\int_{B\cap E_{m,n}}
-1\times 4\epsilon_1(\delta)\,d\Prob\left[\left(X(\omega),
Y(\omega)\right)\right]\geq-4\epsilon_1(\delta).
\end{align*}
Together with \eqref{3.2} this implies the claim. 
\end{proof}

\section{Proof of The Main Theorem}\label{main}

After introducing the tools of Sections \ref{deviations} and 
\ref{ergodic}, the stage is set for a proof of Theorem
\ref{maintheorem}.\\

\begin{proof}
Since $\{\omega:\, U<m\varepsilon\}\cup F_{m,n}^c$ is 
$\sigma(X,Y)$-measurable, we have 
\begin{align*}
0&=\E[\Delta]
=\E\left[\Delta\,\Bigr\|\,U\geq m\varepsilon, F_{m,n}\right]
\times\Prob\left[U\geq m\varepsilon, F_{m,n}\right]+
\int_{\{U<m\varepsilon\}\cup F_{m,n}^c}\hspace{-1cm}\Delta d\Prob
\stackrel{Lem\ref{lem3.1},Lem\ref{lem3.2}}{\geq}\\
&\geq
\left(\left[\frac{1}{2}-3\max\left(\epsilon_4(\delta,\varepsilon)
+2\varepsilon,\epsilon_3(\delta,\varepsilon)\right)\right]\times
\varepsilon+\min\left[\frac{1}{2}-3\left(\epsilon_4(\delta,
\varepsilon)+2\varepsilon\right),
-2\epsilon_2(\delta,\varepsilon)\right]\times(1-\varepsilon)\right)\\
&\hspace{2cm}\times\left(\Prob\left[U\geq m\varepsilon\right]-\Prob
\left[F_{m,n}^c\right]\right)\\
&\quad-\left(4\epsilon_1(\delta)+8\e^
{-nH(\delta)}\right).
\end{align*}
Therefore, 
\begin{multline*}
\Prob\left[U\geq m\varepsilon\right]\\
\leq\frac{4\epsilon_1(\delta)+8\e^{-nH(\delta)}}
{\left[\frac{1}{2}-3\max\left(\epsilon_4(\delta,\varepsilon)+2\varepsilon,
\epsilon_3(\delta,\varepsilon)\right)\right]\times\varepsilon+\min\left[
\frac{1}{2}-3\left(\epsilon_4(\delta,\varepsilon)+2\varepsilon\right),
-2\epsilon_2(\delta,\varepsilon)\right]\times(1-\varepsilon)}\\
+10\e^{-nH(\delta)}
=\frac{{\mathcal O}\left(\delta^{\frac{1}{2}}\right)}
{{\mathcal O}(\varepsilon)+{\mathcal
O}\left(\delta^{\frac{1}{2}}\right)}+{\mathcal O}
\left(\e^{-n\delta}\right).
\end{multline*}
\end{proof}

\hspace{1cm}\\
\vspace{1cm}

\bibliographystyle{plain}
\bibliography{align}

\end{document}